\documentclass{amsart}
\usepackage[pagebackref]{hyperref} 
\usepackage{comment} 
\usepackage{cleveref}  
\hypersetup{colorlinks=true} 

\usepackage[usenames]{color}
\usepackage[all]{xy}
\usepackage[notcite,notref]{}

\usepackage{amsfonts,amsmath,amssymb,amsthm,amscd,amsxtra}
\usepackage{enumerate,epsfig,verbatim}

\usepackage{enumerate,verbatim}
\usepackage{centernot}
\usepackage{todonotes}

\usepackage[T1]{fontenc}
\usepackage[usenames,dvipsnames]{pstricks}
\usepackage[mathscr]{eucal}
\usepackage[notcite,notref]{}
\usepackage{tikz}
\usetikzlibrary{matrix,quotes}
\usepackage[notcite, notref]{}

\usepackage[usenames,dvipsnames]{pstricks}
\usepackage[mathscr]{eucal}
\usepackage{amsfonts,amsmath,amssymb,amsthm,amscd,amsxtra}
\usepackage{tikz-cd} 
\usepackage{enumerate,verbatim}
\usepackage{mathptmx}
\usepackage[notcite, notref]{}

\usepackage{mathtools}
\usepackage{nicefrac}
\usepackage{array}
\usepackage{xcolor}

\usepackage{amsfonts}
\usepackage[margin=1.4in]{geometry}
\usepackage{color}
\usepackage[notcite,notref]{}
\usepackage{url}
\usepackage{datetime}
\usepackage{mathtools}

\usepackage{amssymb}
\usepackage{amsmath}
\usepackage{amsfonts}
\usepackage{amsmath}
\usepackage{amsthm}
\usepackage{amssymb}
\usepackage{amscd}
\usepackage{amsfonts}
\usepackage{amsxtra} 

\usepackage{epsfig}
\usepackage{verbatim}

\SelectTips{cm}{}

\usepackage{amsfonts,amsmath,amssymb,amsthm,amscd,amsxtra}
\usepackage{enumerate,epsfig,verbatim}

\usepackage{amssymb,amstext,amsmath,amscd,amsthm,amsfonts,enumerate,graphicx,latexsym}
\usepackage[usenames]{color}

\usepackage{nicefrac}
\usepackage{array}

\usepackage{comment}
\usepackage{amssymb,amstext,amsmath,amscd,amsthm,amsfonts,enumerate,graphicx,latexsym,color,stmaryrd}
\usepackage{mathptmx}
\usepackage[all]{xy}
\usepackage{comment}

\colorlet{colorSouvik}{orange}

\newtheorem{thm}{Theorem}
\newtheorem{prop}[thm]{Proposition}
\newtheorem{cor}[thm]{Corollary}

\theoremstyle{definition}

\newtheorem{lem}[thm]{Lemma}

\newtheorem{eg}[thm]{Example}

\newtheorem{ques}[thm]{Question}

\newtheorem{rem}[thm]{Remark}

\newtheorem{convention}[thm]{Convention}
\theoremstyle{remark}

\newtheorem*{claim*}{Claim}
\numberwithin{equation}{thm}

\def \rfd {\operatorname{Rfd}}  
\def \Max {\operatorname{MAx}}

\def \Tr {\operatorname {Tr}}

\newcommand{\p}{\mathfrak{p}}
\newcommand{\m}{\mathfrak{m}}

\def\depth{\operatorname{\mathrm{depth}}}

\def\Ext{\operatorname{\mathrm{Ext}}}

\def\Hom{\operatorname{\mathrm{Hom}}}

\def\Max{\operatorname{\mathrm{Max}}}

\def\Tor{\operatorname{\mathrm{Tor}}}

\def\p{\operatorname{\mathfrak{p}}} 
\def\supp {\operatorname{Supp}}
\def\syz{\Omega}
  
\def\Spec{\operatorname{\mathrm{Spec}}}

\def \rfd {\operatorname{Rfd}} 
\def \cdim {\operatorname{\mathsf{CI-dim}}}   
\def \cifd {\operatorname{\mathsf{CI-fd}}}
\def \ucifd {\operatorname{\mathsf{CI*-id}}}
\def \ciid {\operatorname{\mathsf{CI-id}}}

\DeclareMathOperator{\CIdim}{\operatorname{\mathsf{CI-dim}}}
\DeclareMathOperator{\Hdim}{\mathsf{H-dim}}
\DeclareMathOperator{\Gdim}{\mathsf{G-dim}}
\DeclareMathOperator{\pd}{\mathsf{pd}}
\DeclareMathOperator{\cx}{\mathsf{cx}}
 
\DeclareMathOperator{\curv}{\mathsf{curv}}

\DeclareMathOperator{\Cdim}{\mathsf{CI-dim}}

\DeclareMathOperator{\id}{\operatorname{\mathsf{id}}}
\def \n {\mathfrak n}

\title[Homological dimensions]{Homological dimensions, the Gorenstein property, and special cases of some conjectures}       

\author[S. Dey]{Souvik Dey}
\address{Souvik Dey\\ Department of Algebra, Charles University, Faculty of Mathematics and Physics, Sokolovska
83, 186 75, Praha, Czech Republic}
\email{souvik.dey@matfyz.cuni.cz}  

\author[R. Holanda]{Rafael Holanda}
\address{Departamento de Matem\'atica, CCEN, Universidade Federal de Pernambuco, 50740-560, Recife, PE, Brazil}
\email{rf.holanda@gmail.com}  

\author[C. B. Miranda-Neto]{Cleto B. Miranda-Neto}
\address{Departamento de Matem\'atica, Universidade Federal da Para\'iba - 58051-900, Jo\~ao Pessoa, PB, Brazil}
\email{cleto@mat.ufpb.br}   

\date{\today}


\begin{document}

\begin{abstract} Our purpose in this work is multifold. First, we provide general criteria for the finiteness of the projective and injective dimensions of a finite module $M$ over a (commutative) Noetherian ring $R$. Second, in the other direction, we investigate the impact of the finiteness of certain homological dimensions of $M$ if $R$ is local, mainly when $R$ is Cohen-Macaulay and with a partial focus on  duals. Along the way, we produce various freeness criteria for modules. Finally, we give applications, including characterizations of when $R$ is Gorenstein (and other ring-theoretic properties as well, sometimes in the prime characteristic setting), particularly by means of its anticanonical module, and in addition we address special cases of some long-standing conjectures; for instance, we confirm the {\rm 1985} conjecture of Vasconcelos on normal modules in case the module of differentials is almost Cohen-Macaulay.

\end{abstract}

\maketitle

\section{Introduction}


Part of modern homological commutative algebra is concerned with the problem of finding new characterizations, via module theory, of fundamental ring-theoretic properties of commutative rings such as the Gorenstein,  the complete intersection, and the regular properties, among others (it is worth recalling that, in an appropriate setting, such properties also feature a strong geometric counterpart). One of our goals in this paper, for a given  Noetherian local ring $R$, is to address this issue by studying several homological dimensions of a  finitely generated $R$-module $M$, namely, the injective, projective, Gorenstein, and complete intersection dimensions of $M$, and of its algebraic dual -- sometimes iterated with the canonical dual, once $R$ possesses a canonical module -- and then focus on the Gorenstein property of $R$, although we shall be able as well to provide characterizations of all the above-mentioned properties. For the case targeting Gorensteiness, if $R$ is Cohen-Macaulay with canonical module $\omega_R$ then we investigate the case  $M=\omega_R$, particularly via the $R$-dual ${\rm Hom}_R({\omega_R}, R)$, the so-called anticanonical module of $R$, which plays an important role on the properties of $R$ itself; we refer to \cite[Introduction]{anti} for a nice description of this interplay and its connections to other topics.

Furthermore, besides providing criteria for the finiteness of the injective and projective dimensions of $M$ in Section \ref{ringmaps} (for this task, we do not always assume that $R$ is local or Noetherian; the main results in this section are Theorem \ref{thm1} and Theorem \ref{thm3}), and as a fundamental step to achieve the goals described above, we address in Section \ref{dualsetc} the problem of when the finiteness of suitable homological dimensions of the $R$-module ${\rm Hom}_R(M, R)$ forces $M$ to be free, or totally reflexive, or of complete intersection dimension zero, if $R$ is a Cohen-Macaulay local Noetherian ring; one of the main results in this section is Proposition \ref{newhom} and some of its byproducts, e.g., Corollary \ref{corhom} and Proposition \ref{pdhomfinite}. Needless to say, freeness is a classical matter of interest, being in particular the protagonist of the celebrated Auslander-Reiten conjecture, to which we also contribute in one of our results (Corollary \ref{ARcontrib}). See also Corollary \ref {gathered} and Corollary \ref{p}. Criteria for the Gorensteiness of $R$ (mentioned in the previous paragraph) and more on freeness of modules are presented in Section \ref{G}, where the main results are Proposition \ref{Gor-crit-new}, Corollary \ref{ii}, Corollary \ref{charact-p} (this one treats the prime characteristic case) and Corollary \ref{nsyz} (which makes use of syzygies of strongly rigid modules).

Finally, in Section \ref{conjs}, we relate our investigation to issues involving celebrated modules such as modules of differentials and derivations as well as (co)normal modules; more precisely, we address long-standing problems such as the strong form of 
the Zariski-Lipman conjecture \cite{H} about derivation modules, Berger's conjecture \cite{Ber} on differential modules in dimension one, and Vasconcelos' conjecture \cite[p.\,373]{V} about normal modules. More precisely, in the Cohen-Macaulay case, we settle the first one when the differential module is maximal Cohen-Macaulay (Corollary \ref{first}, which in fact proves a more general statement), and the third one when either the conormal module is maximal Cohen-Macaulay (Corollary \ref{second}, where, once again, a more general result is given) or the differential module is almost Cohen-Macaulay (Corollary \ref{second-omega}). As to the second conjecture, we confirm its validity in case the canonical dual of the derivation module has finite projective dimension (Corollary \ref{ber}).

\medskip


\noindent{\it A few conventions and notations.} Throughout this note, unless explicitly stated differently, by a {\it ring} we mean a commutative Noetherian ring with non-zero identity, and by a {\it finite} module over a ring $R$ we mean a finitely generated $R$-module. We denote projective dimension and injective dimension over $R$ by $\pd_R$ and $\id_R$, respectively.

\section{Finiteness criteria for the projective and injective dimensions via ring maps}\label{ringmaps}

The purpose of this first section is to derive general criteria for the finiteness of the projective and injective dimensions of finite modules (not necessarily over local rings) via a suitable ring homomorphism. Our first result is as follows.

\begin{prop}\label{finitecriteria} Let $R\to S$ be a ring homomorphism, with $S$ not necessarily Noetherian, such that $\m S\neq S$ for each $\m \in \Max(R)$. Let $M$ be a finite $R$-module. The following assertions hold true:

\begin{enumerate}[\rm(1)]
    \item If, for each $\n \in \Max(S)$, there exists an integer $h\geq 0$ {\rm (}possibly depending on $\n${\rm )} such that $\Tor^R_h(M, S/\n)=0$, then $\pd_RM<\infty$. 

    \item Assume $S$ is module-finite over $R$. If, for each $\n \in \Max(S)$, there exists an integer $h\geq 0$ {\rm (}possibly depending on $\n${\rm )} such that $\Ext_R^h(M, S/\n)=0$, then $\pd_RM<\infty$.

    \item Suppose $\dim\,M<\infty$, and let $i>\dim\,M$ be an integer. If, for each $\n \in \Max(S)$, we have $\Ext_R^i(S/\n, M)=0$, then $\id_RM<i$ and $R$ is locally Cohen-Macaulay on $\supp_R(M)$.   
    
    \item Suppose $\dim\,R<\infty$. If, for each $\n \in \Max(S)$, we have $\Ext_R^j(S/\n, M)=0$ for all $j\gg 0$, then $\id_RM<\infty$ and $R$ is locally Cohen-Macaulay on $\supp_R(M)$.  
\end{enumerate}
\end{prop}

\begin{proof} (1) For a given $\m \in \Max(R)$, we have $\m S\neq S$ by hypothesis and so there exists $\n \in \Max(S)$ such that $\m S\subseteq \n$. Also, by assumption, there is an integer $h\geq 0$ such that $\Tor_h^R(M, S/\n)=0$. As $S/\n$ is an $R/\m$-vector space (possibly of infinite dimension), we get that $S/\n$ is a direct sum of (possibly infinitely many) copies of $R/\m$, and hence  $\Tor^R_h(M, R/\m)=0$ by \cite[Proposition 7.6]{rot}. So, localizing at each $\m \in \Max(R)$, we obtain $\pd_{R_{\m}}M_{\m}<\infty$. 
Now it is convenient to consider the so-called large restricted flat dimension $\rfd_RM$ of $M$ over $R$, which can be expressed as
$$\rfd_RM=\sup\{\depth R_\mathfrak{p}-\depth_{R_\mathfrak{p}}M_\mathfrak{p} \, \mid \, \mathfrak{p}\in\Spec (R)\};$$ we refer to \cite[Notes 1.6 and (1.0.1)]{ail}. It follows that, for each $\m \in \Max(R)$,  $$\pd_{R_{\m}}M_{\m}=\depth\,R_{\m}-\depth_{R_{\m}}M_{\m}\leq \rfd_RM <\infty,$$ where the finiteness of  $\rfd_RM$ is guaranteed by \cite[Theorem 1.1]{ail}. Hence, using  \cite[Proposition 8.52]{rot}, we conclude $\pd_RM=\sup\{\pd_{R_{\m}}M_{\m} \mid \m \in \Max (R) \}\leq \rfd_RM<\infty$.

\smallskip

(2) Fix $\m \in \Max(R)$. By hypothesis, $\m S\neq S$, hence there exists $\n \in \Max(S)$ such that $\m S\subseteq \n$, and in addition there exists $h\geq 0$ such that $\Ext^h_R(M, S/\n)=0$. As $S$ is module-finite over $R$, we get that $S/\n$ is a direct sum of finitely many copies of $R/\m$, which yields $\Ext_R^h(M,R/\m)=0$ by \cite[Proposition 7.22]{rot}. The rest of the proof is similar to the one given in (1).

\smallskip

(3) By \cite[Proposition 7.21]{rot}, we have $\Ext^i_R(R/\mathfrak{m},M)=0$ for each $\mathfrak{m}\in\Max(R)$. Since $i>\dim\,M\geq \depth_{R_{\m}}M_{\m}$ for each $\m \in \Max(R)\cap \supp_R(M)$, we can localize the above vanishing condition and use \cite[II. Theorem 2]{rob} in order to get $\id_{R_{\m}}M_{\m}<i$. Therefore, for every $R$-module $L$, we have $$\Ext^{\geq i}_R(L, M)_{\m}\cong \Ext^{\geq i}_{R_{\m}}(L_{\m}, M_{\m})=0 \quad \mbox{for\, every} \quad \m \in \Max(R).$$ Thus, $\Ext_R^{\geq i}(L, M)=0$ for every $R$-module $L$, so that $\id_RM<\infty$ and, more precisely, $\id_RM<i$ by \cite[Proposition 3.1.10]{BH93}. Finally, if $\p \in \supp_R(M)$, then the nonzero $R_{\p}$-module $M_{\p}$ has finite injective dimension and hence $R_{\p}$ must be Cohen--Macaulay by the well-known Bass' Theorem.  

\smallskip

(4) As in (3), we have $\Ext^{\gg0}_R(R/\mathfrak{m}, M)=0$ for each $\mathfrak{m}\in\Max(R)$. Localizing this vanishing condition and using \cite[Proposition 3.1.14]{BH93}, we get $\id_{R_{\m}}M_{\m}<\infty$. Thus, $\id_{R_{\m}}M_{\m}= \depth R_{\m}\leq \dim R$ for each $\m \in \Max(R)$, where the equality follows by \cite[Theorem 3.1.17]{BH93}. Hence, for every $R$-module $L$, we have $\Ext^{>\dim R}_R(L, M)_{\m}\cong \Ext^{>\dim R}_{R_{\m}}(L_{\m}, M_{\m})=0$ for every $\m \in \Max(R)$. Thus, $$\Ext_R^{>\dim R}(L, M)=0 \quad \mbox{for\, every\, {\it R}-module {\it L}},$$ and so \cite[Proposition 3.1.10]{BH93} yields $\id_RM\leq {\dim R}<\infty$. The last assertion follows from (3).
\end{proof}

Before we proceed with further preparation for a theorem on finiteness of projective dimension, we record the following byproduct which gives a criterion for an interesting new global bound on injective dimension (in the local case, the bound is attained).

\begin{cor} Let $R\to S$ be a ring homomorphism, with $S$ not necessarily Noetherian, such that $\m S\neq S$ for every $\m \in \Max(R)$. Let $M$ be a finite $R$-module with $\dim\,M<\infty$. 
If, for each $\n \in \Max(S)$, we have $\Ext_R^{{\rm dim}\,M +1}(S/\n, M)=0$, then $$\id_RM \leq {\rm dim}\,M$$ and $R$ is locally Cohen-Macaulay on $\supp_R(M)$. In particular, if $R$ is local then  ${\rm dim}\,M = \id_RM = {\rm depth}\,R$.
    
\end{cor}
\begin{proof} Apply Proposition \ref{finitecriteria}(3) with $i={\rm dim}\,M +1$. If $R$ is local then ${\rm dim}\,M \leq  \id_RM = {\rm depth}\,R$ (see \cite[Theorem 3.1.17]{BH93}), and hence the assertion follows.
\end{proof}

Below we furnish simple instances of a ring homomorphism $R\to S$ satisfying $\mathfrak{m}S\neq S$ for every $\mathfrak{m}\in\Max(R)$.

\begin{eg}\label{instances}

\begin{enumerate}[\rm(1)] 
\item Let $R\to S$ be a finite morphism of rings and assume $(R,\m)$ is local. Then, $\m S\neq S$ is automatic by Nakayama's lemma.

\item Let $R\to S$ be a faithfully flat map. Then, $S/\m S \cong S\otimes_R R/\m \neq 0$ implies $\m S\neq S$ for all $\m \in \Max(R)$.  

\item Let $R\subseteq S$ be an integral extension. Then, each $\m \in \Max(R)$ is the contraction of a maximal ideal of $S$, hence $\m S\neq S$.  

\item Assume $R\subseteq S$ is a ring extension such that $R$ is a retraction of a given ring $S$, i.e., there is an $R$-linear map $S\to R$ whose restriction to $R$ is the identity map (note $R$ is also Noetherian; see, e.g., \cite[Exercise 5.27]{lw}). Then, $\m S\neq S$ for all $\m \in \Max(R)$.

\end{enumerate}
\end{eg}   





    

The next lemma is an elementary fact that will be useful in the sequel.

\begin{lem}\label{11} {\it Let $R$ be a ring. Let $\{T_n\}$ be a covariant homological delta-functor, $\{F^n\}$ a covariant cohomological delta-functor, and $\{G^n\}$ a contravariant cohomological delta-functor, all from the category of $R$-modules to itself. Let $s,h\geq 0$ be integers, and $N,X_0,\ldots,X_h$ be $R$-modules fitting into an exact sequence $$0\to X_h\to\cdots\to X_0\to N\to 0.$$ The following assertions hold true: 

\begin{enumerate}[\rm(1)]
    \item If $T_{s+h-i+1\leq j \leq s+h+1}(X_i)=0$ for all $i=0,\ldots,h$, then $T_{s+h+1}(N)=0$. 

    \item If $F^{s+1\leq j \leq s+i+1}(X_i)=0$ for all $i=0,\ldots,h$, then $F^{s+1}(N)=0$.  

    \item If $G^{s+h-i+1\leq j \leq s+h+1}(X_i)=0$ for all $i=0,\ldots,h$, then $G^{s+h+1}(N)=0$.
\end{enumerate}}
\end{lem}

\begin{proof} The case where $h=0$ is clear, so we assume $h\geq 1$. We break apart the given exact sequence into short exact sequences $$0\to N_i\to X_{i-1}\to N_{i-1}\to 0,$$ for $1\leq i \leq h$ and $R$-modules $N_i$, with $N_h=X_h$ and $N_0=N$. 

(1) Applying $\{T_n\}$ to the short exact sequence above for $i=h$, we obtain $T_{s+2\leq j \leq s+h+1}(N_{h-1})=0$. If $h=1$, we are done, otherwise applying $\{T_n\}$ to the short exact sequence for $i=h-1$, we get $T_{s+3\leq j \leq s+h+1}(N_{h-2})=0$. Continuing similarly, we have $T_{s+h+1}(N_0)=0$, which completes the proof since $N_0=N$. 

\smallskip

(2) Applying $\{F^n\}$ to the short exact sequence above for $i=h$, we get $F^{s+1\leq j \leq s+h}(N_{h-1})=0$. If $h=1$, we are done, otherwise applying $\{F^n\}$ to the short exact sequence for $i=h-1$, we obtain $F^{s+1\leq j \leq s+h-1}(N_{h-2})=0$. Continuing this way, we finally get $F^{s+1}(N_0)=0$, which completes the proof since $N_0=N$. 

\smallskip

(3) The proof is similar to that of (1). 
\end{proof}

Following standard terminology, we say that a module over a ring $S$ is (maximal) Cohen--Macaulay if it is locally (maximal) Cohen--Macaulay everywhere on $\Spec(R)$.

\begin{thm}\label{thm1} Let $R\to S$ be a ring map such that $\m S\neq S$ for each $\m \in \Max(R)$. Let $M$ be a finite $R$-module. Assume that $S$ is Cohen--Macaulay and that any one of the following conditions holds true:

\begin{enumerate}[\rm(1)]
    \item $\Tor^R_{\gg 0}(M,X)=0$ for every maximal Cohen--Macaulay $S$-module $X$ which is locally free on ${\rm Spec}(S)\setminus {\rm Max}(S)$.
    

    \item There exist {\rm (}uniform{\rm )} integers $s,h\geq 0$ such that $\Tor^R_{s+1\leq j \leq s+h+1}(M,X)=0$ for every maximal Cohen--Macaulay $S$-module $X$ which is locally free on ${\rm Spec}(S)\setminus {\rm Max}(S)$.

\item  $S$ is module-finite over $R$ and, in addition, $\Ext_R^{\gg 0}(M,X)=0$ for every maximal Cohen--Macaulay $S$-module $X$ which is locally free on ${\rm Spec}(S)\setminus {\rm Max}(S)$.

      \item $S$ is module-finite over $R$ and, in addition, there exist {\rm (}uniform{\rm )} integers $s,h\geq0$ such that $\Ext_R^{s+1\leq j\leq s+h+1}(M,X)=0$  for every maximal Cohen--Macaulay $S$-module $X$ which is locally free on ${\rm Spec}(S)\setminus {\rm Max}(S)$. 
\end{enumerate}
    Then, $\pd_RM<\infty$. 
\end{thm}

\begin{proof}  For items (1) and (2) (resp. item (3) and (4)), it suffices to prove that for each $\mathfrak{n}\in\Max(S)$ there exists an integer $h'\geq0$ such that $$\Tor^R_{h'}(M, S/\mathfrak{n})=0$$ (resp. $\Ext^{h'}_R(M, S/\mathfrak{n})=0$), by Proposition \ref{finitecriteria}(1) (resp. Proposition \ref{finitecriteria}(2)). So, let $\mathfrak{n}\in\Max(S)$ and set $t=\depth S_\mathfrak{n}(=\dim S_\mathfrak{n})$. If $Z$ stands for the $t$-th syzygy $S$-module of $S/\mathfrak{n}$, then by \cite[Proposition 8.5]{rot} we obtain that $Z_\mathfrak{q}$ is a free $S_\mathfrak{q}$-module for each $\mathfrak{q}\in \Spec (S)\setminus \{\mathfrak{n}\}$, while for $\mathfrak{q}=\mathfrak{n}$ the  $S_\mathfrak{n}$-module $Z_\mathfrak{n}$ is maximal Cohen--Macaulay by \cite[Exercise 1.3.7]{BH93}. This proves that $Z$ is  maximal Cohen--Macaulay and also locally free on ${\rm Spec}(S)\setminus {\rm Max}(S)$. Note that $S$ obviously satisfies this property as well. In other words, there exists an exact sequence
\begin{equation}\label{res-by-MCM} 0\to X_t\to X_{t-1}\to\cdots\to X_0\to S/\mathfrak{n}\to0 \end{equation}
where each $X_i$ is a maximal Cohen--Macaulay $S$-module that is locally free on ${\rm Spec}(S)\setminus {\rm Max}(S)$. Therefore, the existence of the desired integer $h'$ follows by {\it d\'ecalage} for items (1) and (3), and by applying Lemma \ref{11} with $T_*(-):=\Tor^R_*(M,-)$ for item (2) and $F^*(-):=\Ext_R^*(M,-)$ for item (4).
\end{proof}






\begin{rem}\label{regularhyp}
In Theorem \ref{thm1}, if we suppose  $S$ is regular, we may clearly replace all the modules $X$ with $S$ itself. Then, in this case, we note the following:

\begin{enumerate}[\rm(i)]
    \item In item (1), $S$ is a test $R$-module in the sense of \cite{CDT}. In particular, item (1) was proved in \cite[Proposition 2.4]{CDT} if $R\to S$ is a finite local ring map.

\item In item (3), if $R$ is a Cohen-Macaulay local ring with a canonical module $\omega_R$ and $S$ is maximal Cohen-Macaulay over $R$, then $\Hom_R(S,\omega_R)$ is a test $R$-module by \cite[Proposition 3.6]{CDT}.

\end{enumerate}
\end{rem}

\begin{thm}\label{thm3} Let $R\to S$ be a ring map such that $\m S\neq S$ for each $\m \in \Max(R)$. Let $M$ be a finite $R$-module. Assume that $S$ is Cohen--Macaulay and that any one of the following conditions holds true:

\begin{enumerate}[\rm(1)]
    \item $\dim R< \infty$ and $\Ext_R^{\gg 0}(X,M)=0$ for every maximal Cohen--Macaulay $S$-module $X$ which is locally free on ${\rm Spec}(S)\setminus {\rm Max}(S)$.

    \item $\dim S=h< \infty$, $\dim M< \infty$, and there exists a {\rm (}uniform{\rm )} integer $s\geq 0$ with $s+h+1>\dim M$ such that $\Ext_R^{s+1\leq j \leq s+h+1}(X,M)=0$ for every maximal Cohen--Macaulay $S$-module $X$ which is locally free on ${\rm Spec}(S)\setminus {\rm Max}(S)$.
\end{enumerate}
    Then, $\id_RM<\infty$ and $R$ is locally Cohen--Macaulay on $\supp_R(M)$.   
\end{thm}

\begin{proof}
As in the proof of Theorem \ref{thm1}, given $\mathfrak{n}\in\Max(S)$ there exists an exact sequence
(\ref{res-by-MCM}) with $t=\depth\,S_\mathfrak{n}=\dim\,S_\mathfrak{n}$ $(<h$ in item (2)). Item (1) follows from Proposition \ref{finitecriteria}(4), and item (2) from Proposition \ref{finitecriteria}(3) along with Lemma \ref{11}(3) by letting $G^*(-):=\Ext_R^*(-,M)$.
\end{proof}

\begin{rem} As in Remark \ref{regularhyp}, suppose $S$ is regular so that the $X$'s can be replaced with $S$. If in addition $R$ admits a dualizing complex (e.g., if $R$ is a quotient of a Gorenstein ring of finite Krull dimension), then Theorem \ref{thm3}(1) follows from Remark \ref{regularhyp}(i) and \cite[Theorem 3.2]{CDT}.

\end{rem}


\begin{eg} Let $R$ be a reduced ring of dimension $1$. Let $S$ be the integral closure of $R$ in the total quotient ring $Q(R)$ (note $S$ is also Noetherian, by the Krull-Akizuki theorem). Since $S$ is reduced and integrally closed in $Q(R)=Q(S)$, we get that $S$ is normal (see, e.g., \cite[Corollary 2.1.13]{H-S}). As $\dim S=1$ and each localization of $S$ is a normal domain, we obtain that $S$ is regular. Finally, by Example \ref{instances}(3), $\m S \neq S$ for every maximal ideal $\m$ of $R$. Therefore, $R, S$ is a pair as in items (1), (2) and (3) of Theorem \ref{thm1} and as in Theorem \ref{thm3}.  If moreover $R$ is an  analytically unramified local ring, then $S$ is also module-finite over $R$, hence the scenario of Theorem \ref{thm1}(4) also applies. 
\end{eg}



\begin{eg} We give another instance where $R\rightarrow S$ is finite. As recalled in Example \ref{instances}(4), any algebra retraction $R\subseteq S$ satisfies
$\m S\neq S$ for every $\m \in \Max(R)$. In particular, we can take $$R=S^G \quad \mbox{for\, a\, finite} \quad G\leq {\rm Aut}(S)$$ whose order is invertible in $S$. If moreover $S$ is a  domain, then $S$ is module-finite over $R$; see \cite[Exercise 5.27, Exercise 5.28, and Proposition 5.4]{lw}. 
\end{eg}

\begin{eg} Suppose $S$ is Cohen--Macaulay and let $T$ be a polynomial or power series ring over $S$ in the variables $X_1,\ldots, X_n$. Let $J$ be an ideal of $T$ such that $(X_1,...,X_n)J$ is contained in the Jacobson radical of $T$ . Let $R=T/(X_1,\ldots, X_n)J$. Then,  $$R\to R/(X_1,\ldots, X_n)R\cong S$$ are rings as in Theorem \ref{thm1} and Theorem \ref{thm3}. 
\end{eg}

\section{Duals having finite homological dimensions, and freeness}\label{dualsetc}

\subsection{Some preliminaries}\label{preliminaries} We denote by $\Gdim_R$ the Gorenstein dimension over a ring $R$ (see \cite[Definition 12 and Definition 16]{mas}). An $R$-module $M$ is totally reflexive if $\Gdim_RM=0$. Given finite $R$-modules $M, N$, we write $M\approx N$ to mean $M\oplus F \cong N \oplus G$ for some projective $R$-modules $F,G$ (see \cite[Definition 3]{mas}). In this case, $M$ and $N$ are said to be stably isomorphic. Note that if $N$ is projective (resp.\,totally reflexive) and $M\approx N$, then $M$ is also projective (resp.\,totally reflexive). Let $\Tr M$ stand for the Auslander transpose of $M$ (see \cite[Definition 2]{mas}). Recall $M\approx \Tr \Tr M$ (see \cite[Remark (3) following Proposition 4]{mas}). It follows that $M$ is totally reflexive if and only if $\Tr M$  is totally reflexive. Also, note $\Tr M$ fits into an exact sequence \begin{equation}\label{seq-Tr}0\to M^*\to F_0 \to F_1 \to \Tr M \to 0,\end{equation} where $F_0$, $F_1$ are free $R$-modules and $M^*={\rm Hom}_R(M, R)$ is the (algebraic) dual of $M$. In case $R$ possesses a canonical module $\omega_R$, we use the notation $M^{\dagger}={\rm Hom}_R(M, \omega_R)$, the canonical dual of $M$. Thus, $M^{\dagger}\cong M^*$ whenever $R$ is a Gorenstein local ring. 

Next, recall that a finite $R$-module $M$ is said to satisfy $(\widetilde S_n)$, for a given  integer $n\geq 0$, if $$\depth_{R_{\p}}M_{\p}\geq \min\{n, \depth\,R_{\p}\} \quad \mbox{for\, all} \quad \p \in \Spec (R).$$ Since the depth of the zero module is set to be $\infty$ by a widely accepted convention, then in order to check that $M$ satisfies $(\widetilde S_n)$ it suffices to only consider primes in $\supp_R(M)$. Clearly, if $M$ satisfies $(\widetilde S_n)$ then it also satisfies $(\widetilde S_m)$ for all $m<n$. If $n=0$, the condition $(\widetilde S_0)$ trivially holds. For completeness, if $n\geq 1$, one (cohomological) criterion is as follows. If $M$ is $n$-torsionless (i.e., $\Ext^i_R(\Tr M, R)=0$ for all $i=1, \ldots, n$), then
$M$ satisfies $(\widetilde S_n)$ (see \cite[Proposition 11]{mas} or \cite[Propositions (16.30), (16.31) and (16.32)]{BV}).

\subsection{A key proposition and first corollaries} The complete intersection dimension of a finite module $M$ over a local ring $R$ is written $\cdim_RM$. It is related to classical homological dimensions by means of the inequalities $$\Gdim_RM\leq \cdim_RM\leq \pd_RM.$$ We refer to \cite{AGP} for the theory.   For any $n\geq 0$, the $n$-th syzygy module of $M$ over $R$ is denoted ${\rm Syz}^n_RM$, with ${\rm Syz}^0_RM=M$. 

\begin{convention}\label{conv} Whenever convenient, we denote by $\Hdim_RM$ any of the homological dimensions $\pd_RM, \CIdim_RM$ or $\Gdim_RM$.
    
\end{convention}


Our first result in this section is the following.

\begin{prop}\label{dual} Let $R$ be a local ring of depth $t$ and $M$ be a finite $R$-module. Let $r$ be an integer with $0\leq r\leq t$ such that $M$ satisfies $(\widetilde S_r)$ and, if $r<t$, $\Ext_R^{1\le i\le t-r}(M,R)=0$. If $\Hdim_RM^*<\infty$, then $\Hdim_RM=0$. 



\end{prop}  

\begin{proof} First, if $r<t$ then, as $\Ext_R^{1\le i\le t-r}(\Tr \Tr M,R)=0$ (recall $M$ and $\Tr \Tr M$ are stably isomorphic), the module $\Tr M$ is $(t-r)$-torsionless, hence $\depth_R\Tr M \geq t-r$ (see \cite[Proposition 11(c)]{mas}), and the case $r=t$ is trivial. Now notice that $M^*$ is stably isomorphic with ${\rm Syz}^2_R \Tr M$. So, in any of the possibilities for $\Hdim_R$, we obtain $\Gdim_R\Tr M<\infty$. Also, by \cite[Theorem 5.8(1)]{CMSZ}, we get (if $r\geq 1$) $$\Ext_R^{1\leq i \leq r}(\Tr M, R)=0.$$

Next, we give a  proof for each choice of $\Hdim_R$.

\smallskip

(1) The case $\Hdim_RM^*=\pd_RM^*$. Since $M^* \approx {\rm Syz}^2_R \Tr M$, we have $\pd_R\Tr M<\infty$. By the Auslander--Buchsbaum formula, $\pd_R\Tr M=t-\depth_R\Tr M\leq r$. Now, by \cite[p.\,154, Lemma 1(iii)]{mat} we get $\pd_R\Tr M=0$, i.e., $\Tr M$ is $R$-free. Since $M$ is stably isomorphic to $\Tr \Tr M$, it follows that $M$ must be free as well.  

\smallskip

(2) The case $\Hdim_RM^*=\Gdim_RM^*$. By the Auslander--Bridger formula, $$\Gdim_R\Tr M=t-\depth_R\Tr M\leq r.$$ Now, if $r\geq 1$, $\Ext_R^{1\leq i \leq r}(\Tr M, R)=0$ implies $\Gdim_R\Tr M=0$ (see \cite[1.2.7(iii)]{Cr00}), i.e., $\Tr M$ is totally reflexive and therefore $M$ has the same property. The case $r=0$ is clear.

\smallskip

(3) The case $\Hdim_RM^*=\CIdim_RM^*$.  By \cite[Theorem 1.4]{AGP} and item (2) above, we obtain that $M$ is totally reflexive, and hence so is $M^*$. Hence, $\Gdim_R M^*=0$.  By \cite[Theorem 1.4]{AGP}, we then get $$\CIdim_RM^*=\Gdim_RM^*=0.$$  Now, applying \cite[Lemma 3.5]{bjor} we derive  $\CIdim_RM^{**}=0$. Since $M$ is reflexive, we are done.   
\end{proof}

\begin{rem}
When $t=0$ (which forces $r=0$, and note every module satisfies $(\widetilde S_0)$), our Proposition \ref{dual} immediately gives that if $M^*$ is free (resp.\,totally reflexive), then $M$ is also free (resp.\,totally reflexive). Moreover, if $M$ satisfies $(\widetilde S_t)$ and $\pd_RM^*<\infty$ then $M$ must be free. Compare it with \cite[Theorem 3.10]{dle}.
\end{rem}

\begin{rem} It is customary to say that a finite $R$-module $M$ is an {\it ideal-module} if $M^*$ is free; this mimics the standard fact that $I^*\cong R$ whenever $I$ is an ideal with ${\rm grade}\,I\geq 2$. Thus, as a consequence of Proposition \ref{dual} with $t=r$, we obtain that a non-free ideal-module over a local ring of depth $t$ cannot satisfy $(\widetilde S_t)$.
\end{rem}


\begin{cor}\label{specseq}
Let $R$ be a local ring of depth $t$ and $M$ a finite $R$-module such that $\Hdim_RM^*<\infty$. Let $N$ be a non-zero finite $R$-module such that $\Ext^{1\leq i\leq t-1}_R(N,M)=0$. The following assertions hold true:

\begin{enumerate}[\rm(1)]
    \item If $\depth_R\Hom_R(N, M)\geq t$, then $\Hdim_RM=0$.

    \item If $\depth_RN\geq t$ and $\Hom_R(N,M)$ is free, then both $M$ and $N^*$ are free.
\end{enumerate}
\end{cor}

\begin{proof}  Item (2) follows directly from  (1) together with \cite[Corollary 3.7]{dle}. To prove (1), notice that by Proposition \ref{dual} we are done once $M$ satisfies $(\widetilde S_t)$, which we claim to hold. Even more, we will show $\depth_RM\geq t$. To this end, let $\mathfrak{m}$ denote the maximal ideal of $R$ and consider the spectral sequence (see \cite[Proposition 2.1]{FJMS})
$$E_2^{i,j}=H^i_\mathfrak{m}(\Ext^j_R(N,M))\Rightarrow H^{i+j}_\mathfrak{m}(N,M),$$ where $H^{i+j}_\mathfrak{m}(N,M)$ stands for the $(i+j)$-th generalized local cohomology module of the pair $N, M$ (see \cite{H0}, also \cite{suz}).  Since $$E^{i,0}_2=0 \quad \mbox{for\, all} \quad i<t, \quad \mbox{and} \quad E^{i,j}_2=0 \quad \mbox{for\, all} \quad j=1,\ldots,t-1,$$ we conclude by convergence that $H^i_\mathfrak{m}(N,M)=0$ whenever $i<t$, and therefore $\depth_RM\geq t$ by \cite[Theorem 2.3]{suz}.
\end{proof}

\begin{ques}\label{questionhom} Does Corollary \ref{specseq}(2) remain true if the freeness condition on $\Hom_R(N,M)$ is replaced with $\pd_R\Hom_R(N,M)<\infty$? What if $N=M$? A particular positive answer for this question will be given later in Proposition \ref{pdhomfinite}.
    
\end{ques}

\begin{rem}\label{remdepth} We record the following immediate byproduct of the proof of Corollary \ref{specseq} (more precisely, from the spectral sequence argument along with \cite[Theorem 2.3]{suz}). Let $R$ be a local ring and $M, N$ non-zero finite $R$-modules, and set $s:={\rm depth}_R{\rm Hom}_R(N, M)$. If either $s=1$ or $s\geq 2$ and $\Ext^{1\leq i\leq s-1}_R(N,M)=0$, then ${\rm depth}_RM\geq s$. 

    
\end{rem}

For the proposition below, given a finite module $M$ over a local ring $R$, we denote by $\cx_RM$ and $\curv_RM$ the complexity and the curvature of $M$, respectively. For the definitions and properties, see \cite[4.2]{avra0}. It should be noticed that item (4) of the next proposition is an extension of Proposition \ref{dual}.

\begin{prop}\label{newhom} Let $R$ be a local ring of depth $t$, $M$ a finite $R$-module, and $r$ an integer with $0\leq r\leq t$. If $r<t$, suppose $\Ext_R^{1\le i\le t-r}(M,R)=0$. Let $N\neq0$ be a finite $R$-module satisfying  $\pd_R N\leq t-r$. The following assertions hold true: 
\begin{enumerate}[\rm(1)]

\item $\cx_R\Hom_R(M,N)=\cx_RM^*$. 

\item $\curv_R\Hom_R(M,N)=\curv_RM^*$.

\item 
$\Hdim_R\Hom_R(M,N)=\Hdim_RM^*+\pd_RN.$

\item If $M$ satisfies $(\widetilde S_r)$ and $\Hdim_R\Hom_R(M,N)<\infty$, then $\Hdim_RM=0$.






\end{enumerate}
 
\end{prop}  
\begin{proof} First, if $r=t$ then $N$ is free, and hence items (1), (2), and (3) follow easily, whereas (4) is a consequence of Proposition \ref{dual}.

So, we can consider $r<t$. Since $\Ext_R^{1\le i\le t-r}(\Tr \Tr M,R)=0$, the module $\Tr M$ is $(t-r)$-torsionfree and hence $\Tr M $ is stably isomorphic to $\syz^{t-r} X$ for some $R$-module $X$ (see \cite[Proposition 11(b)]{mas}). Now, as $\pd_R N\leq t-r$, we must have $$\Tor^R_{>0}(\Tr M, N)\cong \Tor^R_{>t-r}(X,N)=0.$$ By \cite[Theorem 2.8(b)]{AuB}, we then obtain $M^*\otimes_R N\cong \Hom_R(M,N)$.  Furthermore, $\Tor^R_{>0}(M^*,N)\cong \Tor^R_{>2}(\Tr M, N)=0$, hence $$M^*\otimes_R N\cong M^* \otimes_R^{\mathbf L} N.$$ Also, $\cx_R(N)=0=\curv_R(N)$ (see \cite[Remark 4.2.3]{avra0}).

Now, assertions (1) and (2)  follow from \cite[Proposition 4.2.4(6)]{avra0}. Assertion (3) follows from \cite[Theorem (A.7.6)]{Cr00}, \cite[Theorem 3.11]{cdimcomplex} and \cite[Theorem 5.1]{isw04}] -- in the last reference, we make $R=S=T$ and take into account that $M^*$ and $N$ are Tor-independent, as noticed above  -- and, along with Proposition \ref{dual}, it implies assertion (4).

\end{proof}

\begin{rem} It is possible to avoid the explicit requirement of the Serre-type condition $(\widetilde S_r)$ by replacing it with different hypotheses. Indeed, using \cite[Proposition 11]{mas} or \cite[Proposition (16.30) and Proposition (16.31)]{BV}, all the assertions in Proposition \ref{dual} as well as Proposition \ref{newhom}(4) are seen to be valid whenever $M$ is $r$-torsionless and $\Tr M$ is $(t-r)$-torsionless.
\end{rem}

Proposition \ref{newhom}(4) allows us to extend Corollary \ref{specseq} as follows.

\begin{cor}\label{corhom} Let $R$ be a local ring of depth $t$, $M$ a finite $R$-module, and $r$ an integer with $0\leq r\leq t$ such that, if $r<t$, $\Ext^{1\leq i\leq t-r}_R(M,R)=0$. Suppose there exists a finite $R$-module $N\neq0$ such that $\pd_RN\leq t-r$ and $\Hdim_R\Hom_R(M,N)<\infty$. Also, assume that there exists a finite $R$-module $N'\neq0$ satisfying $\Ext_R^{1\le i \le t-1}(N',M)=0$. The following assertions hold true:
\begin{enumerate}[\rm(1)]
    \item If $\depth_R\Hom_R(N',M)\geq t$, then $\Hdim_RM=0$.
    \item If $\depth_RN'\geq t$ and $\Hom_R(N',M)$ is free, then $M$ and $(N')^*$ are free. 
\end{enumerate}
\end{cor} 
\begin{proof}
We provide two proofs of (1). For the first one, note Proposition \ref{newhom}(3) ensures that $\Hdim_RM^*<\infty$ and therefore Corollary \ref{specseq}(1) applies. As to the second proof, from Remark \ref{remdepth} we derive $\depth_RM\geq t$, so $M$ satisfies $(\widetilde S_t)$ and, consequently, $(\widetilde S_r)$. Thus, Proposition \ref{newhom}(4) applies. Item (2) follows directly from Corollary \ref{specseq}(2).
\end{proof}


\begin{cor}
Let $R$ be a local ring of depth $t$ and $M$ a finite $R$-module such that $\Hdim_RM^*<\infty$ and $\Ext^{1\leq i\leq t-1}_R(M,M)=0$. If $\depth_R\Hom_R(M,M)\geq t$, then $\Hdim_RM=0$. Moreover, if $\Hom_R(M,M)$ is free, then so is $M$.
\end{cor}
\begin{proof} Apply Corollary \ref{corhom} with $N'=M$, $N=R$, and $r=t$.
\end{proof}

\begin{prop}\label{pdhomfinite}
Let $R$ be a local ring of depth $t$, $M$ a finite $R$-module, and $r$ an integer with $0\leq r\leq t$ such that $M$ satisfies $(\widetilde S_r)$ and, if $r<t$, $\Ext^{1\leq i\leq t-r}_R(M,R)=0$. Suppose there exists a finite $R$-module $N\neq0$ such that $\pd_RN\leq t-r$ and $\Hdim_R\Hom_R(M,N)<\infty$. Also, assume that there exists a finite $R$-module $N'\neq0$ satisfying $\Hom_R(N',M)\neq0$ and $\Ext^{1\leq i\leq t-1}_R(N',M)=0$. If $\pd_R\Hom_R(N',M)<\infty$, then both $M$ and $N'$ are free.
\end{prop}
\begin{proof}
By Proposition \ref{newhom}(4), $\Hdim_RM=0$. For any choice of $\Hdim_R$ (see Convention \ref{conv}) we obtain that $M$ is totally reflexive. Thus we apply \cite[Theorem 6.6]{dg} to conclude that $M$ is free. In particular, $\depth_R M=t$ and so \cite[Lemma 3.1(1)(i)]{dg} guarantees that $$\depth_R\Hom_R(N',M)\geq t,$$ while the Auslander-Buchsbaum formula yields that $\Hom_R(N',M)$ is free. So, write $M=R^m$ and $\Hom_R(N',M)=R^h$. It follows that $R^h=\Hom_R(N',M)\simeq(N')^m$ and thus $N'$ is also free.
\end{proof}

\begin{rem}
In virtue of \cite[Theorem 5.18]{dg}, we can suppose in Proposition \ref{pdhomfinite} the condition $$\pd_R\Hom_R(C,\Hom_R(N',M))<\infty,$$ with $C$ a semidualizing $R$-module, instead of $\pd_R\Hom_R(N',M)<\infty$.
\end{rem}

\begin{cor}\label{hdimpdfinite}
Let $R$ be a local ring of depth $t$, $M$ a finite $R$-module such that $\Hdim_RM^*<\infty, \pd_R\Hom_R(M,M)<\infty$, and $\Ext^{1\leq i\leq t-1}_R(M,M)=0$. Suppose there exists an integer $r$ with $0\leq r\leq t$ such that $M$ satisfies $(\widetilde S_r)$ and, if $r<t$, $\Ext^{1\leq i\leq t-r}_R(M,R)=0$. Then, $M$ is free.
\end{cor}

\begin{proof}
Take $N=R$ and $N'=M$ in Proposition \ref{pdhomfinite}.
\end{proof}


The case $r=0$ of Corollary \ref{hdimpdfinite} gives the following contribution to the celebrated Auslander-Reiten conjecture (see \cite{dg}, \cite{cr} and their references on the subject), whose commutative local version predicts that a finite module $M$ over a local ring $R$ must be free if $$\Ext_R^{>0}(M,M\oplus R)=0.$$

\begin{cor}\label{ARcontrib} The Auslander-Reiten conjecture is true if  $\Gdim_RM^*<\infty$ and $\pd_R\Hom_R(M,M)<\infty$. 
    
\end{cor}

Here it should be mentioned that the conjecture holds true provided that $\Gdim_RM<\infty$ and $\pd_R\Hom_R(M,M)<\infty$ (see \cite[Corollary 6.9(2)]{dg}).

In another byproduct, we immediately retrieve the following old result of Vasconcelos.

\begin{cor}{\rm (}\cite[Theorem 3.1]{V1}{\rm )} Let $R$ be a one-dimensional Gorenstein local ring and $M$ a finite $R$-module. If $\Hom_R(M,M)$ is free, then $M$ is free. 
\end{cor} 

The theory developed in Section \ref{ringmaps} enables us to produce freeness criteria for $M$ from the vanishing of (co)homology modules of $\Hom_R(M,N)$ via ring homomorphisms. We gather them together in the following result.

\begin{cor}\label{gathered}
Let $(R,\m)\to S$ be a ring map such that $\m S\neq S$. Let $t$ be the depth of $R$ and $r$ an integer with $0\leq r\leq t$.
Let $M$ be a finite $R$-module satisfying $(\widetilde S_r)$ and, if $r<t$, $\Ext^{1\leq i\leq t-r}_R(M,R)=0$. In addition, let $N\neq 0$ be a finite $R$-module such that $\pd_RN\leq t-r$. Assume that any one of the following conditions holds true:

\begin{enumerate}[\rm(1)]
\item For each $\n \in \Max(S)$, there is an integer $h\geq 0$ {\rm (}possibly depending on $\n${\rm )} such that $\Tor^R_h(\Hom_R(M,N), S/\n)=0$. 

    \item $S$ is module-finite over $R$ and, for each $\n \in \Max(S)$, there is an integer $h\geq 0$ {\rm (}possibly depending on $\n${\rm )} such that $\Ext_R^h(\Hom_R(M,N), S/\n)=0$.

    \item $S$ is Cohen--Macaulay and  $\Tor^R_{\gg 0}(\Hom_R(M,N),X)=0$ for every maximal Cohen--Macaulay $S$-module $X$ which is locally free on ${\rm Spec}(S)\setminus {\rm Max}(S)$.

    \item $S$ is Cohen--Macaulay and  there exist integers $s,h\geq 0$ such that $\Tor^R_{s+1\leq j \leq s+h+1}(\Hom_R(M,N),X)=0$ for every maximal Cohen--Macaulay $S$-module $X$ which is locally free on ${\rm Spec}(S)\setminus {\rm Max}(S)$.

\item $S$ is Cohen--Macaulay, $S$ is module-finite over $R$, and $\Ext_R^{\gg 0}(M^*,X)=0$ for every maximal Cohen--Macaulay $S$-module $X$ which is locally free on ${\rm Spec}(S)\setminus {\rm Max}(S)$. 

      \item $S$ is Cohen--Macaulay, $S$ is module-finite over $R$, and  there exist integers $s,h\geq0$ such that $\Ext_R^{s+1\leq j\leq s+h+1}(\Hom_R(M,N),X)=0$ for every maximal Cohen--Macaulay $S$-module $X$ which is locally free on ${\rm Spec}(S)\setminus {\rm Max}(S)$. 
\end{enumerate}
    Then, $M$ is free. 
\end{cor}

\begin{proof}
Apply Proposition \ref{finitecriteria}(1),(2) for items (1) and (2) respectively, and Theorem \ref{thm1} for items (3)-(6) in order to conclude that $\pd_R\Hom_R(M,N)<\infty$. Finally, we apply Proposition \ref{newhom}(4).
\end{proof}

It is worth recording the following special situation. 


\begin{cor}\label{cor1} Let $R$ be a $d$-dimensional Cohen-Macaulay local ring and $M$ a  maximal Cohen-Macaulay $R$-module. Let $N\neq0$ be a finite $R$-module and $r$ an integer with $0\leq r \leq d$ such that $\pd_RN\leq d-r$ and, if $r<d$, $\Ext^{1\leq i\leq d -r}_R(M,R)=0$. If $\Hdim_R\Hom_R(M,N)<\infty$, then $\Hdim_RM=0$.






 \end{cor} 
 
\begin{proof} Over a Cohen-Macaulay local ring, any maximal Cohen-Macaulay module satisfies $(\widetilde S_n)$  for all $n\geq 0$. Now the result follows from Proposition \ref{newhom}(4). 

\end{proof}

In the Gorenstein case, Corollary \ref{cor1} gives the following result.

\begin{cor}\label{corGor} Let $R$ be a Gorenstein local ring and $M$ a  maximal Cohen-Macaulay $R$-module. If $\Hdim_R\Hom_R(M,N)<\infty$ {\rm (}which is automatic in the $\Gdim_R$ case{\rm )}, where $N\neq0$ is a finite $R$-module with $\pd_RN < \infty$, then $\Hdim_RM=0$. 
\end{cor} 


We shall return to the  Gorenstein property (as a target) in Section \ref{G}.


\begin{cor}\label{first1} Let $(R, \m)$ be a local ring of depth $t$, $M$ a finite $R$-module, and $r$ an integer with $0\leq r\leq t$ such that $M$ satisfies $(\widetilde S_r)$ {\rm (}e.g., take $r=t$ if $R$ is Cohen-Macaulay and $M$ is maximal Cohen-Macaulay{\rm )} and, if $r<t$,  $\Ext_R^{1\le i\le t-r}(M,R)=0$. In addition, suppose there is a ring map $R\to S$, where $S$ is regular and $\m S\neq S$, and a finite $R$-module $N\neq 0$ 
such that $\pd_RN\leq t-r$ and $${\rm Tor}^R_i(\Hom_R(M,N), S)=0 \quad \mbox{for\, all} \quad i\gg 0.$$ Then, $M$ is free.



\end{cor}

\begin{proof} By Remark \ref{regularhyp}(i), we must have $\pd_R\Hom_R(M,N)<\infty$. Now we apply Proposition \ref{newhom}(4).
\end{proof}




Next, we consider the class of Golod local rings that are not hypersurface rings. This includes, for instance, Cohen-Macaulay non-Gorenstein local rings with minimal multiplicity, such as for example the ring $k[\![x, y, z]\!]/(y^2-xz, x^2y-z^2, x^3-yz)$, where $x, y, z$ are formal indeterminates over a field $k$.

\begin{cor}\label{cor-Golod} Let $R$ be a Golod local ring of depth $t$ which is not a hypersurface ring, $M$ a finite $R$-module, and $r$ an integer with $0\leq r\leq t$ such that $M$ satisfies $(\widetilde S_r)$ and, if $r<t$, $\Ext^{1\leq i\leq t-r}_R(M,R)=0$. If there exists a finite $R$-module $N\neq 0$ satisfying $\pd_RN\leq t-r$ and $\Gdim_R\Hom_R(M,N)<\infty$, then $M$ is free. 
\end{cor}



 

\begin{proof} By Proposition \ref{newhom}(4), $M$ must be totally reflexive. But a local ring $R$ as in the statement is known to have the property that every totally reflexive $R$-module is necessarily free (see \cite[Examples 3.5(2)]{AM}). In particular, $M$ is free.
\end{proof}

To close the subsection, let $R$ be a local ring of prime characteristic $p>0$ and let ${\rm F}\colon R\to R$ be the Frobenius map ${\rm F}(a)=a^p$, for all $a\in R$. For a positive integer $e$, we can consider the $e$-th iteration of ${\rm F}$, i.e., the assignment $${\rm F}^e\colon a\mapsto a^{p^e},$$ which defines on $R$ a new $R$-module structure. Such a module is denoted $R^{(e)}$, which is known to play an important role in the theory; e.g., a classical result of Kunz states that $R^{(1)}$ is $R$-flat if and only if $R$ is regular. Also, recall that $R$ is ${\rm F}$-{\it finite} if ${\rm F}^e$ is a finite map, which means $R^{(e)}$  is a finite $R$-module, for some (or equivalently, for every) $e\geq 1$. This property holds for fundamental classes
of rings, for instance, when $R$ is a complete local ring with perfect residue field or a
localization of an affine algebra over a perfect field (see, e.g., \cite[p. 398]{BH93}).


\begin{cor}\label{p} Let $R$ be an {\rm F}-finite local ring of prime characteristic $p>0$ and depth $t$, $M$ a finite $R$-module,  and $r$ an integer with $0\leq r\leq t$ such that $M$ satisfies $(\widetilde S_r)$ and, if $r<t$, $\Ext^{1\leq i\leq t-r}_R(M,R)=0$. Let $N\neq0$ be a finite $R$-module such that $\pd_RN\leq t-r$. Suppose any one of the following assertions:

\begin{enumerate}[\rm(1)]

\item ${\rm Ext}_R^i(\Hom_R(M,N), R^{(e)})=0$ for all $i=\ell, \ldots, \ell + t$ {\rm (}for some integer $\ell \geq 1${\rm )} and infinitely many $e$.

\item ${\rm Tor}^R_j(\Hom_R(M,N), R^{(e)})=0$  for all $j\gg 0$ and $p^e\gg 0$.
\end{enumerate}
Then, $M$ is free.
\end{cor}
\begin{proof} Assume that (1) (resp.\,(2)) holds. Then, 
by virtue of \cite[Corollary 2.4]{N} (resp.\,\cite[Theorem 4.5(1)]{TY}), we deduce that $\pd_R\Hom_R(M,N)<\infty$. Now Proposition \ref{newhom}(4) finishes the proof.
\end{proof}

More on the positive characteristic case will be provided in Subsection \ref{char=p}.

\section{The Gorenstein property, and more on freeness}\label{G}

This section is mainly concerned with criteria for the Gorensteiness of Cohen-Macaulay local rings admitting a canonical module, and for the freeness of finite modules over such rings. We will also present criteria for regular and complete intersection local rings.

\subsection{A general criterion} Our first result in this part makes use of an iteration between the algebraic and the canonical duals of a given module, and requires the resulting module to have finite projective dimension.

\begin{prop}\label{Gor-crit-new}
Let $R$ be a $d$-dimensional Cohen-Macaulay local ring  possessing a canonical module. Let $M$ be a finite $R$-module such that $M^*\neq 0$ and $\pd_R(M^*)^{\dagger}<\infty$. The following assertions hold true:

\begin{enumerate}[\rm(1)]

\item If $d\leq 1$, then $M^{**}$ is free.

\item If $\Ext^j_R(M,R)=0$ for all  $j=1,\ldots,d$ when $d\geq 1$, 
or for all $j>0$ when $d=0$, then $M$ is free.
\end{enumerate} In either case, $R$ is Gorenstein.
\end{prop}






\begin{proof} (1) Note $\depth_RM^*\ge \min\{2,\depth\,R\}=\min\{2, d\}=d$. Thus, the $R$-module $M^*$ is maximal Cohen-Macaulay, and hence so is its canonical dual $(M^*)^{\dagger}\neq 0$. As $\pd_R(M^*)^{\dagger}<\infty$, we obtain $(M^*)^{\dagger}\cong R^{\oplus b}$ for some integer $b\geq 1$. Being $M^*$ maximal Cohen-Macaulay, we deduce that $$M^*\cong (M^*)^{\dagger \dagger}\cong \omega_R^{\oplus b},$$ where $\omega_R$ stands for the canonical module of $R$. Consequently, we have an exact sequence $0\to \omega_R^{\oplus b}\to F_0\to F_1$ for some finite free $R$-modules $F_0,F_1$. This, in turn, gives a short exact sequence \begin{equation}\label{seq}0\to \omega_R^{\oplus b}\to F_0\to X \to 0,\end{equation} for an $R$-submodule $X\subset F_1$. As $d\le 1$, we get that $X$ is necessarily maximal Cohen-Macaulay, and hence $\Ext^1_R(X,\omega_R)=0$, which implies that the sequence (\ref{seq}) splits. Thus, $\omega_R^{\oplus b}$ is a direct summand of $F_0$, and so $\omega_R$ itself is a direct summand of $F_0$. It follows that $\omega_R$ is free, i.e., $R$ is Gorenstein. For the freeness of the bidual of $M$, we now have $(M^*)^{\dagger}\cong M^{**}$, which must be maximal Cohen-Macaulay, hence free because 
$\pd_RM^{**}<\infty$.

\smallskip

(2) It follows from \cite[Theorem 2.1(i)]{cr} (also \cite[Remark 2.2(i)]{cr} in the Artinian case) that $$(M^*)^\dagger \cong M\otimes_R\omega_R$$ and that this module is maximal Cohen-Macaulay, hence free. Therefore, both $\omega_R$ and $M$ must be free, whence the result.
\end{proof}

\begin{cor}\label{d=1torsionless}
Let $R$ be a Cohen--Macaulay local ring of dimension $d\leq 1$ possessing a canonical module. If $M\neq0$ is a torsionless finite $R$-module such that $\pd_R(M^*)^\dagger<\infty$, then $M$ is free.
\end{cor}
\begin{proof}
Proposition \ref{Gor-crit-new}(1) ensures that $M^{**}$ is free and $R$ is Gorenstein. Now, \cite[Theorem 17(b)]{mas} yields
$$\Ext^2_R(\Tr M,R)=0,$$ and therefore $M$ is reflexive by \cite[Proposition 5]{mas}.
\end{proof}

Corollary \ref{d=1torsionless} motivates us to ask the following.

\begin{ques} Let $R$ be a Cohen--Macaulay local ring possessing a canonical module, and let $s$ be a positive integer such that ${\rm dim}\,R\leq s$.  If $M\neq0$ is an $s$-torsionless finite $R$-module such that $\pd_R(M^*)^\dagger<\infty$, then must $R$ be Gorenstein and $M$ be free?
\end{ques}


\begin{rem}\rm For $R$ as above, assume in addition that $R$ is locally Gorenstein on its punctured spectrum and that $M$ is maximal Cohen-Macaulay with a rank. Then, by \cite[Lemma 2.1]{HH}, the condition $\Ext^j_R(M,R)=0$ for all $j=1,\ldots,d$ (present in Proposition \ref{Gor-crit-new}(2) in case $d\geq 1$) is equivalent to $M\otimes_R\omega_R$ being maximal Cohen-Macaulay.
    
\end{rem}

\begin{rem}\rm Maximal Cohen-Macaulay modules of finite injective dimension are said to be {\it Gorenstein} modules. Now for completeness we record the following statement, which again assumes no constraint on the dimension of the ring. Let $R$ be as above, $M^*\neq 0$ and $\pd_R(M^*)^{\dagger}<\infty$. If $$\Ext^i_R(M, R)=0 \quad \mbox{for\, all} \quad i=1, \ldots, {\rm max}\{1, d-2\},$$ then $M^*$ is Gorenstein. Indeed, applying \cite[Lemma, p.\,2763]{JD} we deduce that $M^*$ is maximal Cohen-Macaulay, and exactly as in the proof of Proposition \ref{Gor-crit-new}(1) there is an integer $b\geq 1$ such that $M^*\cong  \omega_R^{\oplus b}$, which is a Gorenstein module.
    \end{rem}

\subsection{The case of the (anti)canonical module}

Now we focus on the case where $M=\omega_R$. Note $M^*={\omega_R^*}$ is the anticanonical module of $R$, which (as mentioned in the introduction) is also of recognized significance.

We begin by invoking part of a motivating question from \cite{cr}.

\begin{ques}{\rm (}\cite[Question 5.24]{cr}{\rm )} \label{quest} \rm Let $R$ be a Cohen-Macaulay local ring with canonical module $\omega_R$.  If $\pd_R\omega_R^*<\infty$ or $\pd_R(\omega_R^*)^{\dagger}<\infty$, must $R$ be Gorenstein?
    
\end{ques}

Because $\omega_R$ is maximal Cohen-Macaulay, the first half of this question is immediately seen to have an affirmative answer by Corollary \ref{cor1}. Alternatively, it also follows from some results of \cite{cr}; indeed, since $\pd_R\omega_R^*<\infty$ we obtain $\pd_R\Tr\omega_R<\infty$ and so $$\Ext^{d+1}_R(\Tr\omega_R,R)=0.$$ If $d=0$, we apply \cite[Corollary 3.21]{cr} with $M=\omega_R$. Inductively, we may suppose that $\omega_R$ is a vector bundle (i.e., $R$ is locally Gorenstein on its punctured spectrum), so that \cite[Corollary 3.27]{cr} applies.

We are able, however, to provide a much stronger statement with $\Hdim_R\Hom_R(\omega_R,N)<\infty$ (along with some extra hypotheses), for a suitable finite $R$-module $N\neq0$, in place of $\pd_R\omega_R^*<\infty$. Note this extends even the case $\Gdim_R\omega_R^*<\infty$. In addition, we settle the second half of Question \ref{quest} in dimension at most 1. Below in Corollary \ref{ii} we record such facts, but first we need the following auxiliary lemma, which is probably known to the experts.

  \begin{lem}\label{iv} A local ring $R$ is Gorenstein if and only if $R$ admits a non-zero finite module $M$ such that $\Gdim_RM<\infty$ and $\id_RM<\infty$.  
\end{lem}
\begin{proof}
	If $R$ is Gorenstein, then we can simply take $M = R$. To prove the converse, let $M \neq 0$ be a finite $R$-module with $\Gdim_RM$ and $\id_RM$ both finite. Since $\Gdim_RM<\infty$, we can use Gorenstein dimension approximation (see \cite[Lemma 2.17]{CFH06}) to guarantee the existence of a short exact sequence \begin{equation}\label{seq-lemma}0 \to M \to X \to Y \to 0,\end{equation} where $\pd_RX<\infty$ and $Y$ is totally reflexive. Hence, either $Y=0$ or (by the Auslander-Bridger formula) $\depth_RY=\depth\,R$.  It now follows from Ischebeck's theorem (see \cite[3.1.24]{BH93}) that $\Ext_R^1(Y,M)=0$, which implies that the sequence (\ref{seq-lemma}) splits, i.e., $X \cong M\oplus Y$. Therefore, since $\pd_RX<\infty$, we get $\pd_RM<\infty$. Now, being $\pd_RM$ and $\id_RM$ both finite, we can apply \cite[Corollary 4.4]{Fo77} to conclude that $R$ is Gorenstein.
\end{proof}

\begin{cor} Let $R$ be a local ring of depth $t$, and $r$ an integer with $0\leq r\leq t$. Let $M$ be a  non-zero finite $R$-module satisfying $(\widetilde S_r)$ and, if $r<t$, $\Ext^{1\leq i\leq t-r}_R(M,R)=0$. If $\id_RM^*<\infty$ and there exists a finite $R$-module $N\neq0$ such that $\pd_RN\leq t-r$ and $\Hdim_R\Hom_R(M,N)<\infty$, then $R$ is Gorenstein and $M$ is free.
\end{cor}
\begin{proof}
First, Proposition \ref{newhom}(4) gives $\Hdim_RM=0$. In particular, $M^*\neq0$ and $\Gdim_RM^*<\infty$, and so by Lemma \ref{iv} we conclude that $R$ must be Gorenstein. In particular, $\pd_RM^*<\infty$ by \cite[Exercise 3.1.25]{BH93} and therefore Proposition \ref{dual} ensures that $M$ is free. (Instead of applying Proposition \ref{dual}, we can also argue directly by noticing that $M^*$ is also totally reflexive, hence free, and so is $M$.)
\end{proof}

\begin{cor}\label{ii} Let $R$ be a $d$-dimensional Cohen-Macaulay local ring  admitting a canonical module $\omega_R$. Assume any one of the following situations:
\begin{enumerate}[\rm(1)]

\item $d\leq 1$ and $\pd_R(\omega_R^*)^{\dagger}<\infty$.

\item $d\geq 2, \pd_R(\omega_R^*)^\dagger<\infty$ and $\Ext^j_R(\omega_R,R)=0$ for all  $j=1,\ldots,d$.

\item There exist an integer $r$ with $0\leq r\leq d$ and a finite $R$-module $N\neq0$ such that $\pd_RN\leq t-r$, $\Hdim_R\Hom_R(\omega_R,N)<\infty$ and, in case $r<d$, $\Ext^{1\leq i\leq d-r}_R(\omega_R,R)=0$.
\end{enumerate}
Then, $R$ is Gorenstein.
\end{cor} 

\begin{proof} Assertions (1) and (2) follow readily from Proposition \ref{Gor-crit-new} with $M=\omega_R$. Now let us assume (3). By Corollary \ref{cor1} with $M=\omega_R$ (which is maximal Cohen-Macaulay), we obtain that $$\Gdim_R\omega_R=\Hdim_R\omega_R=0<\infty.$$ Since moreover $\id_R\omega_R<\infty$, we conclude by Lemma \ref{iv} that $R$ is Gorenstein, as needed.
\end{proof}

\begin{rem} For completeness, it is worth recalling the following fact shown in \cite[Corollary 2.2]{HH} as a partial solution to the famous Tachikawa conjecture (see \cite[Conjecture 1.2]{cr}). Let $R$ be as above and in addition suppose $R_{\mathfrak p}$ is Gorenstein for every ${\mathfrak p}\in {\rm Spec}\,R$ with ${\rm height}\,{\mathfrak p}= 0$. If  $\Ext^i_R(\omega_R,R)=0$ for all  $i=1,\ldots,d$, then $R$ is Gorenstein. 
There is also the question (see \cite[Question 4.7]{cr}), which can be regarded as a dual version of Tachikawa's conjecture, as to whether a Cohen-Macaulay local ring $R$ of dimension $d\geq 1$ with canonical module $\omega_R$ must be Gorenstein if $${\rm Ext}^j_R(\omega_R^*,R)=0 \quad \mbox{for\, all} \quad j>0.$$
\end{rem}





\subsection{Positive characteristic}\label{char=p} In this part, we stick to the preparation given for the statement of Corollary \ref{p} in order to provide two more results in the prime characteristic setting.

\begin{cor}\label{pGor} Let $R$ be an {\rm F}-finite local ring of prime characteristic and depth $t$. If there exist an integer $r$ with $0\leq r\leq t$ and a finite $R$-module $N\neq0$ satisfying $\pd_RN\leq t-r$, $\Gdim_R\Hom_R(R^{(e)},N) <\infty$ for some $e\geq 1$, and, in case $r<t$, $\Ext^{1\leq i\leq t-r}_R(R^{(e)},R)=0$, then  $R$ is Gorenstein.
\end{cor}
\begin{proof}  As Frobenius push-forward localizes, we can write $\depth_{R_{\p}}(R^{(e)})_{\p}=\depth_{R_{\p}} R_{\p}^{(e)}=\depth 
 R_{\p}$ for all $\p\in \Spec(R)$, where the last equality follows from the fact that a sequence $\{\xi_1, \ldots, \xi_t\}\subset \p R_{\p}$ is an $R^{(e)}$-sequence if and only if $$\{\xi_1^{p^e}, \ldots, \xi_t^{p^e}\}$$ is an $R$-sequence, with $p={\rm char}\,R$. It follows that $R^{(e)}$ satisfies $(\widetilde S_n)$ for every $n$. Now, Proposition \ref{newhom}(4) gives $\Gdim_RR^{(e)}=0<\infty$, and we are done by \cite[Theorem 6.2]{TY}. 
\end{proof}

\begin{cor}\label{charact-p} Let $R$ be an {\rm F}-finite local ring of prime characteristic and depth $t$. Let $r$ be an integer with $0\leq r\leq t$. The following assertions hold true:

\begin{enumerate}[\rm(1)]

\item If there exists a finite $R$-module $N\neq0$ such that $\pd_RN\leq t-r$ and $\pd_R\Hom_R(R^{(e)},N)<\infty$ for some $e>0$, and if $\Ext^{1\leq i\leq t-r}_R(R^{(e)},R)=0$ in case $r<t$, then $R$ is regular.

\item  If there exists a finite $R$-module $N\neq 0$ such that $\id_R N <\infty$ and $\id_R \Hom_R(R^{(e)}, N)<\infty$ for some $e>0$, then $R$ is regular.  

    \item  If there exists a finite $R$-module $N\neq0$ such that $\pd_RN\leq t-r$ and $\Cdim_R\Hom_R(R^{(e)},N)<\infty$ for some $e>0$, and if $\Ext^{1\leq i\leq t-r}_R(R^{(e)},R)=0$ in case $r<t$, then $R$ is a complete intersection ring.
\end{enumerate}
\end{cor}

\begin{proof} We shall use the fact (noticed in the proof of Corollary \ref{pGor}) that ${\rm depth}_RR^{(e)}={\rm depth}\,R$ and consequently $R^{(e)}$ satisfies $(\widetilde S_n)$ for every $n$. 

\smallskip

(1) By Proposition \ref{newhom}(4), we have  $\pd_RR^{(e)}=0<\infty$. Now,  since the Frobenius map ${\rm F}\colon R\to R$ is a contracting endomorphism, we can apply \cite[Theorem 1.1]{ahiy} to conclude that $R$ is regular.

\smallskip

(2) Since $\id_R N<\infty$, we can make use of a well-known result of Ischebeck (see, e.g., \cite[Exercise 3.1.24]{BH93}), which yields $\Ext_R^{>0}(R^{(e)},N)=0$, and so $\mathbf{R}\Hom_R(R^{(e)},N)\cong \Hom_R(R^{(e)},N)$. Thus, $$\id_R \Hom_R(R^{(e)},N)=\id_R \mathbf{R}\Hom_R(R^{(e)},N)=\pd_RR^{(e)}+\id_R N,$$ where the last equality follows by \cite[Theorem (A.7.7), (A.7.4.1) and  (A.7.5.1)]{Cr00}. Thus, $\pd_RR^{(e)}<\infty$ and we are done by  \cite[Theorem 1.1]{ahiy}.  

\smallskip

(3) By Proposition \ref{newhom}(4), we have $\Cdim_RR^{(e)}=0<\infty$ and therefore $\Gdim_RR^{(e)}$ must be finite as well. By \cite[Theorem 6.2]{TY}, we obtain that $R$ is Gorenstein. Note that, because $R$ is {\rm F}-finite, there is an equality between $\Cdim_RR^{(e)}$ and the complete intersection flat dimension  $\cifd_RR^{(e)}$ of $R^{(e)}$ (see \cite[Definition 2.4]{Keri}). Consequently, $\cifd_RR^{(e)}<\infty$. Now let $\ucifd_RR^{(e)}$ denote the upper complete intersection injective dimension of $R^{(e)}$ (see \cite[Definition 2.6]{Keri}). Since $R$ is Gorenstein, we can apply \cite[Corollary 5.7]{Keri} to get $$\ucifd_RR^{(e)}<\infty.$$ By \cite[Remark 2.8]{Keri}, the complete intersection injective dimension $\ciid_RR^{(e)}$  must also be finite. Finally, by \cite[Theorem C, p.\,2595, or Corollary 6.7]{Keri}, the ring $R$ is necessarily a complete intersection.


\end{proof}

\begin{rem} An observation on the case $e=1$ is that in \cite[Proposition 1]{BM} it was proved that $R$ must be a complete intersection ring if $\Cdim_RR^{(1)}<\infty$ (without taking the $R$-dual). Furthermore, we can ask whether Corollary \ref{charact-p} remains valid if (in items (1) and (3)) the algebraic duals are replaced with canonical duals, if $\omega_R$ exists.
    
\end{rem}

\subsection{Further criteria}

We close the section with new criteria for the regular, complete intersection (e.g., hypersurface) and Gorenstein properties of local rings. Following \cite[Definition 2.1]{Dao-et-al}, a finite module $N$ over a local ring $(R,\m, k)$ is said to be {\it strongly rigid} provided that ${\pd}_RM<\infty$ whenever $M$ is a finite $R$-module with $${\rm Tor}^R_i(M, N)=0 \quad \mbox{for\, some} \quad i\geq 1.$$ For instance, if $N$ is either the $R$-module $k=R/\m$ or any integrally closed $\m$-primary ideal of $R$, or if $N$ has infinite projective dimension over the local ring $R  =  ({\mathbb Z}/p{\mathbb Z})[\![x, y, z]\!]/(xy-z^2)$ for a prime number  $p\geq 3$, then $N$ is a strongly rigid $R$-module.

\begin{cor}\label{nsyz} Let $R$ be a local ring of depth $t$ and let $r$ be an integer with $0\leq r\leq t$. Let $N$ be a  strongly rigid $R$-module such that 
 $\Ext^{r+1\leq i\leq t}_R(N,R)=0$ if $r<t$, and let $N'\neq 0$ be a finite $R$-module such that $\pd_RN'\leq t-r$. The following assertions hold true:  

\begin{enumerate}[\rm(1)]

\item If  
$\pd_R\Hom_R({\rm Syz}_R^rN,N')<\infty$, then $R$ is regular.


\item If $\CIdim_R\Hom_R({\rm Syz}_R^rN,N')<\infty$, then $R$ is a complete intersection ring.

\item If $\Gdim_R\Hom_R({\rm Syz}_R^rN,N')<\infty$, then $R$ is Gorenstein.

\item If  $\Gdim_R\Hom_R({\rm Syz}_R^rN,N')<\infty$ and $R$ is Golod, then $R$ is a hypersurface ring.

\end{enumerate}
\end{cor}

\begin{proof} First, a couple of observations that will be useful to the proof. The $R$-module ${\rm Syz}_R^rM$ satisfies $(\widetilde S_r)$ for any given finite $R$-module $M$ (see \cite[Exercise 1.3.7]{BH93}). In particular, if  $r=t$ then ${\rm Syz}_R^tM$  satisfies $(\widetilde S_t)$. Otherwise, if $r<t$, the hypothesis $\Ext^{r+1\leq i\leq t}_R(N,R)=0$ (along with induction) implies $$\Ext^{1\leq i\leq t-r}_R({\rm Syz}_R^rN,R)=0.$$ 



(1) By Proposition \ref{newhom}(4), we conclude that ${\rm Syz}_R^rN$ is free, hence $\pd_RN<\infty$. Therefore, ${\rm Tor}^R_i(k, N)=0$ for all $i\gg 0$. Because $N$ is strongly rigid, this forces $\pd_Rk<\infty$, i.e., $R$ is regular.

\smallskip

(2) By Proposition \ref{newhom}(4), we obtain that $\CIdim_R{\rm Syz}_R^rN=0$, hence $\CIdim_RN<\infty$ by \cite[Lemma (1.9)]{AGP}. Now, applying \cite[Corollary 3.4.4]{T} (and observing that strongly rigid modules are {\it test} modules), we conclude that $R$ is a complete intersection.

\smallskip

(3) By Proposition \ref{newhom}(4), the module ${\rm Syz}_R^rN$ must be totally reflexive. Therefore, $\Gdim_RN<\infty$ and hence \cite[Corollary 3.9]{CS} forces $R$ to be Gorenstein.    

\smallskip

(4) According to (3) above, $R$ is Gorenstein. Now we use the fact that Gorenstein Golod local rings are necessarily hypersurface rings (see \cite[Remark after Proposition 5.2.5]{avra0}).
\end{proof}

Taking $r=t$ and $N=k$ (the residue class field of $R$) in the corollary above, we immediately derive the following result. 

\begin{cor}\label{nsyz-k} Let $R$ be a local ring with residue field $k$ and depth $t$. The following assertions hold true:  

\begin{enumerate}[\rm(1)]

\item If  
$\pd_R({\rm Syz}_R^tk)^*<\infty$, then $R$ is regular.


\item If $\CIdim_R({\rm Syz}_R^tk)^*<\infty$, then $R$ is a complete intersection ring.

\item If $\Gdim_R({\rm Syz}_R^tk)^*<\infty$, then $R$ is Gorenstein.

\item If  $\Gdim_R({\rm Syz}_R^tk)^*<\infty$ and $R$ is Golod, then $R$ is a hypersurface ring.

\end{enumerate}
\end{cor}






\begin{eg} The results above can be used to detect instances of (dual) modules having infinite Gorenstein dimension. A plentiful source of examples comes from the class consisting of almost complete intersection rings that are not complete intersection rings. Indeed, consider for instance the local ring $R=k[\![x_1, \ldots, x_m]\!]/I$, where $x_1, \ldots, x_m$ are formal indeterminates over a field $k$ and $I$ is an ideal of height $h$ which is minimally generated by $h+1$ elements. By \cite[Corollary 1.2]{K0}, $R$ cannot be Gorenstein. So, assuming for simplicity that $R$ is Cohen-Macaulay, Corollary \ref{nsyz-k}(3) yields $$\Gdim_R({\rm Syz}_R^{m-h}k)^*=\infty.$$

    
\end{eg}

\section{Applications to some conjectures}\label{conjs}

In this last part we consider some long-held conjectures involving certain modules such as derivation modules, differential modules, and normal modules, which are needless to say important entities in commutative algebra and algebraic geometry.

\subsection{Strong Zariski-Lipman conjecture on derivation modules} Also dubbed Herzog-Vasconcelos-Zariski-Lipman conjecture, it predicts that $R$ must be regular if ${\pd}_R{\rm Der}_k(R)<\infty$, where $R$ is either
\begin{equation}\label{two-classes}
{k}[x_1, \ldots, x_m]_{\mathfrak q}/I \ \ \ ({\mathfrak q}\in {\rm Spec}\,{k}[x_1, \ldots, x_m])\quad \mbox{or} \quad {k}[\![x_1, \ldots, x_m]\!]/I,\end{equation} with $I$ a proper radical ideal and $x_1, \ldots, x_m$ indeterminates over a field ${k}$ of characteristic 0 (see the survey \cite{H}). As usual, ${\rm Der}_k(R)$ stands for the module of $k$-derivations of $R$, i.e., the additive maps $R\to R$ that vanish on $k$ and satisfy Leibniz rule (more generally, given any $R$-module $N$ we can consider the module ${\rm Der}_k(R, N)$ formed by the $k$-derivations of $R$ with values in $N$). Now recall that, in both situations, $R$ admits a universally finite $k$-differential module, which is designed to be a finite $R$-module, denoted by $\Omega_{R/k}$. In the first case, $\Omega_{R/k}$ is just the module of K\"ahler differentials of $R$ over $k$. See \cite{K} for the general theory.

Our contribution is as follows (notice that it substantially improves \cite[Corollary 5.8(iii)]{cr}).

\begin{cor}\label{first} Let $R$ be as in {\rm (}\ref{two-classes}{\rm )} and write $t={\rm depth}\,R$. Then, the strong Zariski-Lipman conjecture holds true if there exists an integer $r$ with $0\leq r\leq t$ such that $\Omega_{R/k}$ satisfies $(\widetilde S_r)$ and, if $r<t$, $$\Ext^{1\leq i\leq t-r}_R(\Omega_{R/k},R)=0.$$ In particular, the conjecture is true when $R$ is Cohen-Macaulay and $\Omega_{R/k}$ is maximal Cohen-Macaulay.
\end{cor}
\begin{proof} Recall that
   $\Omega_{R/k}^*\cong {\rm Der}_k(R)$ (see, e.g., \cite[p.\,192]{mat}). Now we apply  Proposition \ref{dual} to obtain that $\Omega_{R/k}$ is free. But this is equivalent to $R$ being regular -- for a proof of this fact in case $R={k}[x_1, \ldots, x_m]_{\mathfrak q}/I$ (resp. $R={k}[\![x_1, \ldots, x_m]\!]/I$), see \cite[Theorem 7.2]{K} (resp. \cite[Theorem 14.1]{K}).
\end{proof}

\begin{rem} Maintain the above setup and notations, and recall that in fact $\Hom_R(\Omega_{R/k}, N)\cong {\rm Der}_k(R, N)$ for any $R$-module $N$ (see \cite[p.\,192]{mat}). Now, using Proposition \ref{newhom}(4), our Corollary \ref{first} is immediately seen to admit the following generalization. Let $r$ be an integer with $0\leq r\leq t$ such that $\Omega_{R/k}$ satisfies $(\widetilde S_r)$ and, if $r<t$, $\Ext^{1\leq i\leq t-r}_R(\Omega_{R/k},R)=0$. Suppose there exists a finite $R$-module $N\neq 0$ satisfying  $\pd_R N\leq t-r$ and $$\pd_R{\rm Der}_k(R, N)<\infty.$$ Then, $R$ is regular (to retrieve Corollary \ref{first}, take $N=R$). It is also worth pointing out that a similar argument can be used to generalize Corollary \ref{derconj} below.
\end{rem}

The above remark suggests a possible generalization of the conjecture, as follows.

\begin{ques} Let $R$ be as in {\rm (}\ref{two-classes}{\rm )}. Suppose there exists a finite $R$-module $N\neq 0$ satisfying $\pd_RN<\infty$ and $\pd_R{\rm Der}_k(R, N)<\infty$. Is it true that $R$ must be regular?
    \end{ques}

It is well-known that the condition ${\pd}_R{\rm Der}_k(R)<\infty$ forces $R$ to be a normal domain, which in particular settles the conjecture in the case of curves. A major case is that of quasi-homogeneous complete intersections with
isolated singularities (see \cite[Theorem 2.4]{H}). Notice furthermore that if $t\leq 2$ then, since ${\rm Der}_k(R)$ is a dual, the conjecture is easily seen to be equivalent to the original version of the Zariski-Lipman conjecture, which says that $R$ is regular if ${\rm Der}_k(R)$ is free. The critical open case of the latter is when $R$ is Cohen-Macaulay (in fact, Gorenstein) of dimension 2; so, as a consequence of Corollary \ref{first}, we obtain that it holds true provided that ${\rm depth}\,\Omega_{R/k}=2$.  

Let us mention that there is also the following related conjecture (see \cite[Conjecture 3.12]{cleto}, also \cite[Conjecture 5.10]{cr}). If $R$ is as in (\ref{two-classes}) and  $$\Gdim_R{\rm Der}_k(R) <\infty \quad \mbox{(resp.\,}\Cdim_R{\rm Der}_k(R) <\infty\mbox{)},$$ then $R$ is a Gorenstein ring (resp.\,a complete intersection ring). In this regard, we have the following result.

 \begin{cor}\label{derconj} Let $R$ be as in {\rm (}\ref{two-classes}{\rm )} and write $t={\rm depth}\,R$. Suppose 
 there exists an integer $r$ with $0\leq r\leq t$ such that $\Omega_{R/k}$ satisfies $(\widetilde S_r)$ and, if $r<t$, $\Ext^{1\leq i\leq t-r}_R(\Omega_{R/k},R)=0$ {\rm (}e.g., if $R$ is Cohen-Macaulay and $\Omega_{R/k}$ is maximal Cohen-Macaulay{\rm )}. The following assertions hold:
 \begin{enumerate}[\rm(1)] 
\item If $\Gdim_R{\rm Der}_k(R)<\infty$, then $\Omega_{R/k}$ is totally reflexive. If in addition $R$ is Golod, then $R$ is a hypersurface ring.

\item If $\Cdim_R{\rm Der}_k(R)<\infty$ and either ${\rm Ext}_R^{2j}({\rm Der}_k(R), {\rm Der}_k(R))=0$ or ${\rm Ext}_R^{2j}(\Omega_{R/k}, \Omega_{R/k})=0$ for some integer $j\geq 1$, then $R$ is regular.
 \end{enumerate}  
 \end{cor}  
\begin{proof} (1) The first part follows by Proposition \ref{dual}. For the second part, assume by way of contradiction that $R$ is not a hypersurface ring. In particular, $R$ cannot be regular. On the other hand,  Corollary \ref{cor-Golod} forces $\Omega_{R/k}$ to be free, which as we recalled in the proof of Corollary \ref{first} is equivalent to $R$ being regular, a contradiction.

\smallskip

(2) Let us first consider the case where ${\rm Ext}_R^{2j}({\rm Der}_k(R), {\rm Der}_k(R))=0$. Since $\Cdim_R{\rm Der}_k(R)<\infty$, we can apply \cite[Theorem 4.2]{AB} to get $\pd_R{\rm Der}_k(R)<\infty$, and the result follows by Corollary \ref{first}. Next, suppose ${\rm Ext}_R^{2j}(\Omega_{R/k}, \Omega_{R/k})=0$. Notice that Corollary \ref{dual} yields $$\CIdim_R\Omega_{R/k}=0.$$ Using  \cite[Theorem 4.2]{AB} once again, we obtain $\pd_R\Omega_{R/k}<\infty$, hence $\pd_R\Omega_{R/k}=\CIdim_R\Omega_{R/k}=0$ and therefore $R$ is regular.
   \end{proof}

\subsection{Berger's conjecture on differential modules} Pick $R$ as in (\ref{two-classes}). The 1963 Berger's conjecture  (see \cite{Ber}) asserts that $R$ must be regular if  $\dim\,R=1$ and $\Omega_{R/k}$ is torsionfree. We refer to \cite{H} for further information about this problem.

\begin{cor}\label{ber} Let $R$ be as in {\rm (}\ref{two-classes}{\rm )}. Then, Berger's conjecture holds true if $$\pd_R{\rm Der}_k(R)^{\dagger}<\infty.$$
\end{cor}
\begin{proof} In the present setup, being a reduced local ring of dimension 1, $R$ is  Cohen-Macaulay. Its differential module $\Omega_{R/k}$, being torsionfree by hypothesis and possessing a generic rank  (which for completeness is equal to $\dim\,R=1$), must also be torsionless (see \cite[Exercise 1.4.18]{BH93}). In addition, we have $(\Omega_{R/k}^*)^{\dagger}\cong {\rm Der}_k(R)^{\dagger}$. Now we are in a position to apply Corollary \ref{d=1torsionless} with $M=\Omega_{R/k}$ to conclude that $\Omega_{R/k}$ is free; as recalled in the proof of Corollary \ref{first}, this means that $R$ is regular.
\end{proof}

\begin{rem} A few comments are in order. Berger's conjecture is known to be true if 
$\pd_R{\rm Der}_k(R)<\infty$. This is because in this case the local ring $R$ must be normal, and hence, being one-dimensional, necessarily regular. However, and inspired by Corollary \ref{ber} above, we wonder whether the conjecture holds true provided that 
$$\pd_R{\rm Der}_k(R)^*<\infty,$$ where we recall that the module ${\rm Der}_k(R)^*$ (the bidual of $\Omega_{R/k}$) has been considered in the literature and is dubbed {\it module of Zariski differentials of $R$ over $k$} (see, e.g., \cite{P}).
\end{rem}

\subsection{Vasconcelos' conjecture on normal modules}  In this last subsection, let $R=S/I$ where $S$ is a local ring and $I$ is an ideal with $\pd_SI<\infty$ (typically, $S$ is taken regular). Note this setting is far more general than (\ref{two-classes}). There is a conjecture by Vasconcelos (see \cite[p.\,373]{V}) which states that $I$ must be generated by a regular sequence if  $\pd_{R}{\rm N}_{R}<\infty$, where ${\rm N}_{R}$ is the normal module of $R$, i.e.,
$${\rm N}_{R}={\rm Hom}_{R}(I/ I^2, R)=(I/I^2)^*.$$

\begin{cor}\label{second} Let $R$ be as above, and write $t={\rm depth}\,R$. Then, Vasconcelos' conjecture holds true if there exists an integer $r$ with $0\leq r\leq t$ such that $I/I^2$  satisfies $(\widetilde S_r)$ and, if $r<t$, $$\Ext^{1\leq i\leq t-r}_R(I/I^2 ,R)=0.$$ In particular, the conjecture is true when $R$ is Cohen-Macaulay and $I/I^2$ is maximal Cohen-Macaulay.
\end{cor}
\begin{proof} By Proposition \ref{dual}, we obtain that $I/I^2$ is free. Since $\pd_SI<\infty$, this forces $I$ to be generated by an $S$-sequence (see \cite{V0}).
    \end{proof}

\begin{rem} Maintain the above setup and notations. Using Proposition \ref{newhom}(4), our Corollary \ref{second} is immediately seen to admit the following generalization. Let $r$ be an integer with $0\leq r\leq t$ such that $I/I^2$ satisfies $(\widetilde S_r)$ and, if $r<t$, $\Ext^{1\leq i\leq t-r}_R(I/I^2,R)=0$. Suppose there exists a finite $R$-module $N\neq 0$ satisfying  $\pd_R N\leq t-r$ and $$\pd_R{\rm Hom}_{R}(I/ I^2, N)<\infty.$$ Then, $I$ is generated by an $S$-sequence (to retrieve Corollary \ref{second}, pick $N=R$). 
\end{rem}

This remark seems to suggest the following potential generalization of Vasconcelos' conjecture.

\begin{ques} Let $R$ be as above. Suppose there exists a finite $R$-module $N\neq 0$ satisfying $\pd_RN<\infty$ and $\pd_R{\rm Hom}_{R}(I/ I^2, N)<\infty$. Is it true that $R$ must be a complete intersection ring?
    \end{ques}

Next, we exemplify the Cohen-Macaulay case of Corollary \ref{second} in its contrapositive form, that is, we proceed to illustrate the property $$\pd_{R}{\rm N}_{R}=\infty.$$ It is worth mentioning that in this case the normal module can benefit, in particular, from the well-established theory of infinite free resolutions (see \cite{avra0}), which includes, e.g., the investigation of eventual periodicity of resolutions as well as connections to complexity, curvature, and other numerical invariants of modules.   

\begin{eg} Let $k$ be a field and
$R=S/I=k[x, y, z]_{(x, y, z)}/I$, where $$I=(x^{\ell}, xy^{\ell -2}z, y^{\ell -1}z), \quad \ell \geq 3.$$ Note $R$ is a (non-Gorenstein) Cohen-Macaulay almost complete intersection  local ring. Moreover, $S/I^2$ has the same feature; this has been observed in \cite[Proposition 2.5]{monom} by means of the theory of Buchberger graphs, but here we provide a much simpler argument. It suffices to notice that, for any given $\ell \geq 3$, a minimal free resolution of $I^2$ is given by
$0\to S^5\stackrel{\varphi}{\to} S^6 \to I^2\to 0$, where
    $$\varphi ~ = ~\left(\begin{array}{ccccc}
0  & 0 & 0 & 0 & -y^{\ell -2}z\\
0 & 0 & -y & 0 & x^{\ell -1}\\
0  & 0 & x & -y^{\ell -3}z & 0\\
0  & -y & 0 & x^{\ell -2}& 0\\
-y  & x & 0 & 0 & 0\\
x  & 0 & 0 & 0 & 0
\end{array}\right),$$ which, by the classical Hilbert-Burch theorem, yields the claim. Thus, by the short exact sequence $$0\to I/I^2\to S/I^2\to S/I\to 0$$ we deduce that the conormal module $I/I^2$ is maximal Cohen-Macaulay. Consequently, Corollary \ref{second} gives $\pd_{R}{\rm N}_{R}=\infty$.
\end{eg}

\begin{eg} Let $c$ be an integer with $5\leq c\leq 10$, and let $J$ be the homogeneous ideal of the standard graded polynomial ring ${\mathbb Q}[x_0, \ldots, x_c]$ defining a set of $$1+c+ \left\lceil\frac{c(c-1)}{6}\right\rceil$$ general points in $c$-dimensional projective space over ${\mathbb Q}$ (here, as usual, $\left\lceil {q}\right\rceil$ denotes the smallest integer which is bigger than, or equal to, a given number $q\in {\mathbb Q}$). So, e.g., if $c=5$ then $J$ defines $10$ general points in ${\mathbb P}_{\mathbb Q}^5$. Now consider the regular local ring $S={\mathbb Q}[x_0, \ldots, x_c]_{(x_0, \ldots, x_c)}$ and the ideal $I=JS$. Then, according to \cite[paragraph after Conjecture 7.2]{MX}, the quotient $R=S/I$ must be a non-Gorenstein (hence $I$ is not generated by an $S$-sequence) Cohen-Macaulay local ring and the $R$-module $I/I^2$ is (necessarily maximal) Cohen-Macaulay. By virtue of Corollary \ref{second}, we deduce $\pd_{R}{\rm N}_{R}=\infty$.
\end{eg}

We close the paper by establishing the following result, related to the module of differentials. Recall that a finite $R$-module $M$ is {\it almost Cohen-Macaulay} if ${\rm depth}_RM\geq {\rm dim}\,R -1$.

\begin{cor}\label{second-omega} Let $R$ be as in {\rm (}\ref{two-classes}{\rm )}, and suppose $R$ is Cohen-Macaulay. Then, Vasconcelos' conjecture holds true if $\Omega_{R/k}$ is almost Cohen-Macaulay.
\end{cor}
\begin{proof} Since $I$ is radical and ${\rm char}\,k=0$, the so-called conormal sequence takes the form 
\begin{equation}\label{con}0\to I/I^{(2)}\to R^m\to \Omega_{R/k}\to 0,\end{equation} where
$I^{(2)}$ denotes the second symbolic power of $I$ (i.e., the ideal formed by the $g\in S$ such that $gf\in I^2$ for some $R$-regular element $f$).
This exact sequence shows, by means of the standard depth lemma, that $I/I^{(2)}$ is maximal Cohen-Macaulay since $\Omega_{R/k}$ is almost Cohen-Macaulay. On the other hand, dualizing the short exact sequence $$0\to I^{(2)}/I^2\to I/I^2\to I/I^{(2)}\to 0$$ and using that $I^{(2)}/I^2$ is in fact the $R$-torsion of $I/I^2$, we obtain $$(I/I^{(2)})^*\cong (I/I^2)^*={\rm N}_R.$$ Now, applying Corollary \ref{cor1}, we deduce that 
$I/I^{(2)}$ is free. By the sequence (\ref{con}), this implies $\pd_R\Omega_{R/k}\leq 1$, which according to \cite[Remark (b), p.\,374]{V} forces $I$ to be generated by a regular sequence, as needed.  \end{proof}

Clearly, $\Omega_{R/k}$ is  almost Cohen-Macaulay (with $R$ as in Corollary \ref{second-omega}) if, for example,  ${\rm dim}\,R=2$ and $\Omega_{R/k}$ is torsionfree, or if ${\rm dim}\,R=3$ and $\Omega_{R/k}$ is reflexive. In particular, the latter situation suggests the question as to when $\Omega_{R/k}$ is reflexive, and we refer to \cite[Remark 4, p.\,10]{P}  for a list of instances where this property holds.

\bigskip

\noindent{\bf Acknowledgements.} The first-named author was partially supported by the Charles University Research Center program No. UNCE/24/SCI/022 and a grant GACR 23-05148S from the Czech Science Foundation. The second-named author was supported by CNPq (grants 200863/2022-3 and 170235/2023-8). The third-named author was partially supported by CNPq (grants 406377/2021-9 and 313357/2023-4).

\end{document}